\documentclass[12pt,oneside]{amsart}
\usepackage[utf8]{inputenc}
\usepackage{graphicx,xcolor,hyperref,microtype,enumitem,multicol}
\usepackage[a4paper,scale=0.7]{geometry}
\newcommand{\R}{\mathbb{R}}
\newcommand{\Rr}{\mathcal{R}}
\newcommand{\B}{\mathcal{B}}
\newcommand{\F}{\mathcal{F}}
\newcommand{\Ss}{\mathcal{S}}
\newcommand{\M}{\mathcal{M}}
\newcommand{\E}{\mathcal{E}}
\DeclareMathOperator{\Sp}{Sp}
\DeclareMathOperator{\spn}{span}
\DeclareMathOperator{\pos}{pos}

\providecommand{\norm}[1]{\lVert#1\rVert}

\newtheorem{theorem}{Theorem}[section]
\newtheorem{definition}[theorem]{Definition}
\newtheorem{prop}[theorem]{Proposition}
\newtheorem{lemma}[theorem]{Lemma}

\title{A four-dimensional body of constant width}
\author[M.G. Mercado-Flores]{Marcela G. Mercado-Flores}
\address[M.G. Mercado-Flores]{Centro de Ciencias Matemáticas, UNAM Campus Morelia, Morelia, Mexico}
\email{mmercado@matmor.unam.mx}
\author[M. Raggi]{Miguel Raggi}
\address[M. Raggi]{ENES, UNAM Campus Morelia, Morelia, Mexico}
\email{mraggi@gmail.com}
\author[E. Roldán-Pensado]{Edgardo Roldán-Pensado}
\address[E. Roldán-Pensado]{Centro de Ciencias Matemáticas, UNAM Campus Morelia, Morelia, Mexico}
\email{eroldan@matmor.unam.mx}

\begin{document}
	
	\begin{abstract}
		The study of bodies of constant width is a classical subject in convex geometry, with the $3$-dimensional Meissner bodies being canonical examples. This paper presents a novel geometric construction of a body of constant width in $\R^4$, addressing the challenge of constructing such bodies in higher dimensions. Our method produces a natural analogue of the second Meissner body, by modifying a 4-dimensional Reuleaux simplex. The resulting body possesses tetrahedral symmetry and has a boundary composed of both smooth surfaces and a non-smooth subset of the Reuleaux 4-simplex.
        
        Furthermore, we analyze the orthogonal projection of this body onto the 3-dimensional hyperplane of its base. This ``shadow'' is a 3-dimensional body of constant width with tetrahedral symmetry. It has six elliptical edges and its volume is only slightly larger than that of the Meissner bodies. This body was recently constructed as a projection of a different 4-dimensional body, however the construction presented here is new and gives additional properties.
	\end{abstract}
	
	\maketitle
	
	\section{Introduction}
	
	\subsection{History}
	
	Bodies of constant width are fundamental objects in convex geometry that have been extensively studied in dimensions two and three, due to their rich structure and numerous applications in areas such as engineering, optimization, and graph theory. A convex body in $\R^n$ is said to have constant width if the distance between every pair of parallel supporting hyperplanes is the same in all directions. This concept has been thoroughly explored in both classical and modern literature (see, for example, \cite{CG1983}, \cite{HM1993}, \cite{MMO2019}).
	
	In $\R^3$, classical examples include bodies of revolution obtained from symmetrical bodies of constant width in dimension two, the Meissner bodies \cite{Mei1918} which are obtained from the Reuleaux tetrahedron by modifying certain edges by ``rounding them'', as well as the infinite family of constant width bodies constructed in \cite{MR2016}. These bodies have been studied not only for their theoretical significance, but also because of their rich combinatorial structure, which is explored in \cite{MPRR2020}, as well as their connections to metric and volume related problems.
	
	Depending on which edges are rounded, two distinct types of Meissner bodies can be obtained. To construct the first type, the three edges of the Reuleaux tetrahedron that meet at a common vertex are rounded. To obtain the second type, the three edges belonging to one of its faces are rounded instead.
	
	More recently, several researchers have explored the possibility of extending these objects to higher dimensions. A notable example is the construction developed in \cite{AMO2023}, where the notion of a \emph{peabody}, a special type of body of constant width in $\R^3$, was introduced. These bodies make use of the so-called focal conics, which possess key properties that facilitate the construction by replacing the singularities of the Reuleaux tetrahedron with subarcs of focal conics. This construction was generalized to $\R^4$ in \cite{AMO2024}, where the authors introduce focal quadrics and construct a body of constant width in four dimensions that exhibits all the symmetries of the 4-dimensional simplex. Notably, this body is unique. That is, it does not define a family of bodies of constant width.
	
	In \cite{LO2007}, bodies of constant width in higher dimensions are constructed from those in lower dimensions. However, this construction depends on arbitrary intersections and provides little information about the structure of the resulting bodies.

    More recently, in \cite{ABNPR2025}, a simple construction of a body of constant width $2$ is given for every dimension $n\ge 2$. The volume of this body is shown to be strictly less than $(0.9)^n$ times that of the $n$-dimensional unit ball.
	
	\subsection{A new body of constant width}
	
	The present work aims to provide a more explicit construction for a constant width body in four dimensions that emphasizes the geometric structure of the body, proposing a variation of the approaches found in the literature. Our contribution is based on a geometric approach inspired by the $3$-dimensional case, but adapted to four dimensions.
	
	The definitions are made precise in Section \ref{sec:def}.
	
	\begin{theorem}\label{thm:main}
		Let $\Rr$ be a unit $4$-dimensional Reuleaux simplex with vertices $A,B,C,D,E$. There is a set $\M_0$ containing $\{A,B,C,D,E\}$ and contained in the $2$-skeleton of $\Rr$ (formed by parts of $2$-dimensional spheres) such that:
		\begin{enumerate}
			\item The body $\M$ obtained as the intersection of all unit balls centered at points in $\M_0$ is of constant width.
			\item $\M_0$, and therefore $\M$, is fixed by any isometry that fixes $E$ and permutes $A,B,C,D$.
			\item $\M_0$ is contained in the boundary of $\M$.
			\item Every point in $\partial\M\setminus\M_0$ is smooth.
			\item Every diameter of $\M$ has at least one endpoint in $\M_0$.
		\end{enumerate}
	\end{theorem}
	
	The body $\M$, as stated in Theorem \ref{thm:main}, can be seen as a natural analogue of the second Meissner body. In the $3$-dimensional case, this Meissner body is typically described by modifying the three edges forming a single triangular face of the Reuleaux tetrahedron. Equivalently, this corresponds to preserving the three edges that meet at the vertex opposite that face. In fact, the Meissner body can be constructed as the intersection of all unit balls whose centers lie on these three \emph{preserved} edges (see Figure \ref{fig:Meissner}). In the $4$-dimensional case, the generating set $\M_0$ plays the role of these preserved edges. The modifications to the initial Reuleaux simplex $\Rr$ are performed in a neighborhood of the ``base'' facet with vertices $\{A,B,C,D\}$. Correspondingly, the preserved set $\M_0$ is composed primarily of portions of the 2-skeleton associated with the opposing apex, $E$.
	
	\begin{figure}
		\centering
		\includegraphics{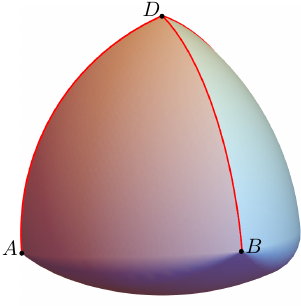}
		\caption{The second Meissner body, the non-smooth circular edges are marked in red.}
		\label{fig:Meissner}
	\end{figure}
	
	\subsection{The shadow of \texorpdfstring{$\M$}{M}}
	
	Projecting the body $\M$ onto the affine hyperplane spanned by the vertices $\{A,B,C,D\}$ yields a $3$-dimensional body of constant width, which we refer to as the \emph{shadow} of $\M$. By construction, this shadow body inherits the symmetries of the regular tetrahedron; that is, it is invariant under any isometry that permutes the vertices $\{A,B,C,D\}$. We believe this shadow body is of independent interest due to its unique geometric properties.
    
    Since this shadow body has non-smooth edges, it cannot be constructed as a Minkowski average of previously known bodies of constant width with smooth edges, such as Meissner bodies or peabodies. However, this body was previously constructed in \cite{ABPR2025} as a projection of the body introduced in \cite{ABNPR2025}. Given that their method of construction differs substantially from ours, we prove that the two bodies coincide and establish some of its properties.
	
	Visually, the shadow body closely resembles the Reuleaux tetrahedron (see Figure \ref{fig:shadow}). Its boundary, like that of the Meissner bodies, is partially composed of patches of spheres. However, it is distinguished by its six elliptical edges and the non-spherical surfaces in their vicinity. For any two vertices $X, Y \in \{A,B,C,D\}$, the corresponding edge is the shorter arc of the unique ellipse with minimal area that is centered at the centroid of the tetrahedron and passes through $X$ and $Y$. Notably, these ellipses each have an eccentricity of $1/\sqrt{2}$, resulting in edges with a lower curvature than their circular counterparts.

	\begin{figure}
		\centering
		\includegraphics{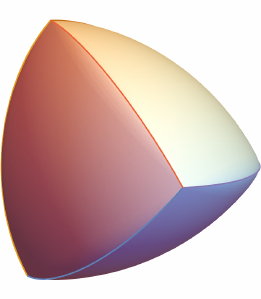}
		\caption{The shadow of $\M$.}
		\label{fig:shadow}
	\end{figure}
	
	\subsection{Paper structure}
	
	The paper is organized as follows. In Section \ref{sec:notation}, we establish the notation to be used throughout the paper. In Section \ref{sec:def}, we present the detailed construction of the $4$-dimensional body of constant width $\M$. The proof that $\M$ has constant width, along with the proofs of the remaining properties stated in Theorem \ref{thm:main}, is presented in Section \ref{sec:constw}. Finally, in Section \ref{sec:shadow}, we analyze the $3$-dimensional body of constant width obtained as the shadow of $\M$.
	
	\section{Notation}\label{sec:notation}
	
	To facilitate readability, we summarize here the notation that will be used throughout the rest of the paper.
	
	\begin{itemize}
		\item $\B(x, r)$: Closed ball of radius $r$ centered at the point $x$.
		\item $\Ss(x, r)$: Sphere of radius $r$ centered at $x$; that is, the boundary of $\B(x, r)$.
		\item $\spn(X)$: The linear space spanned by a point set $X$.
		\item $\pos(X)$: The positive cone spanned by point set $X$.
		\item $X^\circ$: The interior of a set $X$ or the relative interior of a lower-dimensional set $X$.
		\item $A, B, C, D, E$: The vertices of a regular 4-dimensional simplex.
		\item $G$: The centroid of the simplex $ABCDE$.
		\item $O$: The origin of the coordinate system, chosen such that the vectors from $O$ to $A$, $B$, $C$ and $D$ form an orthogonal basis.
		\item $M_{X,Y}$: The midpoint of the segment joining points $X$ and $Y$.
		\item $M_{X,Y,Z}$: The centroid of the triangle with vertices $X$, $Y$ and $Z$.
		\item $\Rr$: The 4-dimensional Reuleaux simplex with vertices $A, B, C, D, E$.
		\item $F_{X,Y}$: A 1-dimensional circular arc in $\Rr$ between vertices $X$ and $Y$.
		\item $F_{X,Y,Z}$: A $2$-dimensional spherical face of $\Rr$ determined by $X$, $Y$, and $Z$.
		\item $F_{X,Y,Z,W}$: A $3$-dimensional spherical face of $\Rr$ determined by $X$, $Y$, $Z$, and $W$.
		\item $\M$: The main convex body under study in this paper.
		\item $\F_{X,Y}$: A special subset of the 2-face $F_{X,Y,E}$, used in the construction of $\M$.
		\item $S_{X,Y}$: The circular arc on the plane $\spn(\{X,Y\})$ that, together with $F_{X,E}$ and $F_{Y,E}$, forms the boundary of $\F_{X,Y}$.
		\item $\M_0$: A distinguished subset of the boundary $\partial \M$, defined as the union of the $\F_{X,Y}$ with $X,Y\in\{A,B,C,D\}$ and $X\neq Y$.
		\item $\Sp^k(S, r)$: The spindle defined as the intersection of all closed balls of radius $r$ centered at points of a subset $S$ of a $k$-dimensional sphere.
	\end{itemize}
	
	\section{Definition of the Body} \label{sec:def}
	
	We begin with a regular $4$-simplex with vertices $A, B, C, D, E$ and side length $1$.  
	For convenience, we assume that the vectors $\{A, B, C, D\}$ form an orthogonal basis, each with norm $1/\sqrt{2}$, and that
	\[
	E = \frac{\varphi}{2} (A + B + C + D),
	\]
	where $\varphi = \frac{1 + \sqrt{5}}{2}$ is the golden ratio.
	
	Let 
	\[
	G = \frac{A + B + C + D + E}{5}
	\]
	denote the centroid of the simplex.
	Note that the vectors $G$ and $E$ are aligned.
	
	The Reuleaux simplex $\Rr$ with vertex set $\{A, B, C, D, E\}$ is defined as
	\[
	\Rr = \bigcap_{X \in \{A, B, C, D, E\}} \B(X,1).
	\]
	
	This body shares the same face structure as the regular simplex: its $k$-dimensional faces are $k$-spherical patches, for $k\in\{0,1,2,3\}$. The geometric and combinatorial properties of $\Rr$ have been thoroughly studied and are well understood (see \cite{MMO2019}).
	
	Consider a 2-dimensional face of $\Rr$ incident to vertex $E$, for example, $F_{A,B,E}$. This face lies in the 3-dimensional subspace $\spn(\{A, B, E\})$.
	
	The plane $\spn(\{A, B\})$ divides $F_{A,B,E}$ into two connected regions. We define $\F_{A,B}$ as the region containing vertex $E$. In other words, $\F_{A,B}$ is the subset of $F_{A,B,E}$ lying on the same side of the plane $\spn(\{A, B\})$ as $E$. See Figure~\ref{fig:Fab} for an illustration.
	
	\begin{figure}
		\centering
		\includegraphics{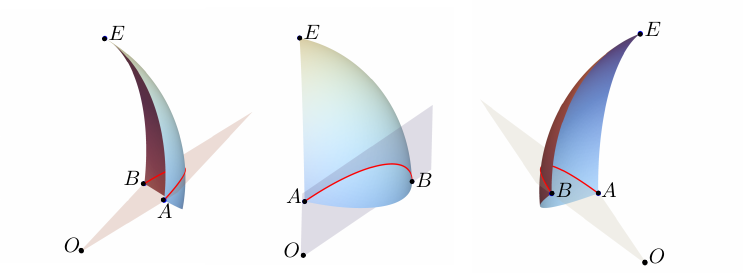}
		\caption{Three views of the face $F_{A,B,E}$ being cut by the plane $\spn(\{A, B\})$ in order to produce $\F_{A,B}$, as seen in the $3$-space generated by $A$, $B$ and $E$.}
		\label{fig:Fab}
	\end{figure}
	
	Note that $F_{A,B,E}$, and therefore $\F_{A,B}$, is contained in the sphere $\Ss(M_{C,D}, \sqrt{3}/2)$, where $M_{C,D}$ denotes the midpoint of $C$ and $D$, since $F_{A,B}$ is contained in the intersection of $\Ss(C,1)$ and $\Ss(D,1)$.
	
	Furthermore, the set $\F_{A,B}$ is bounded by three circular arcs. Two of these are $F_{A,E}$ and $F_{B,E}$, which form the boundary of the 2-face $F_{A,B,E}$. The arc $F_{A,E}$ is contained in the set of points equidistant from $B$, $C$, and $D$, which is precisely $\spn(\{A,E\})$. Likewise, $F_{B,E}$ lies in the plane $\spn(\{B,E\})$. We denote the third arc, which lies in the plane $\spn(\{A,B\})$, by $S_{A,B}$.
	
	Hence, we can express
	\[
	\F_{A,B} = \pos(\{A, B, E\}) \cap \Ss(M_{C,D}, \sqrt{3}/2),
	\]
	where $\pos(\{A,B,E\})$ denotes the positive cone generated by the vectors $A, B,$ and $E$.
	
	Analogously, we define the subsets $\F_{X,Y}$ for the rest of the distinct pairs $X,Y \in \{A,B,C,D\}$, which in turn define the corresponding boundary arcs $S_{X,Y}$.
	
	Let
	\[
	\M_0 = \bigcup_{\substack{X,Y \in \{A,B,C,D\} \\ X \neq Y}} \F_{X,Y},
	\]
	and define our body as
	\begin{align*}
		\M &= \bigcap_{P \in \M_0} \B(P, 1)\\
		&= \bigcap_{\substack{X,Y \in \{A,B,C,D\} \\ X \neq Y}} \bigcap_{P \in \F_{X,Y}} \B(P, 1).
	\end{align*}
	
	Note that, with this definition of $\M$, part (2) of Theorem \ref{thm:main} is satisfied. We prove the remaining parts in the following section.
	
	\section{\texorpdfstring{$\M$}{M} is a body of constant width} \label{sec:constw}
	
	This section is dedicated to proving that $\M$ is a body of constant width and establishing the remaining properties outlined in Theorem \ref{thm:main}. Our proof strategy is organized as follows:
	\begin{itemize}
		\item We begin by showing that the diameter of the generating set $\M_0$ is exactly 1. A direct consequence of this and the construction of $\M$ is that $\M_0$ must be a subset of the boundary $\partial\M$, which proves part (3) of Theorem \ref{thm:main}.
		
		\item Next, we establish a crucial property: for every boundary point $P \in \partial\M\setminus\M_0$, there exists a unique point $Q \in \M_0$ such that the distance $\norm{P-Q}=1$.
		
		\item As a consequence of this uniqueness, we prove two key results: that every point in $\partial\M \setminus \M_0$ is smooth, and that every diameter of $\M$ must have at least one endpoint in $\M_0$. This shows that parts (4) and (5) of Theorem \ref{thm:main} are satisfied.
		
		\item Finally, we use the preceding properties to show that $\M$ has constant width, which completes the proof of Theorem \ref{thm:main}.
	\end{itemize}
	
	\subsection{\texorpdfstring{$\M_0$}{M0} is a subset of \texorpdfstring{$\M$}{M}}\label{sub:diameter}
	First, note that it is enough to show that the diameter of $\M_0$ is 1. Indeed, assume that the diameter is 1. If there is a point $P\in\M_0\setminus \M$, then by the definition of $\M$, there is a point $Q \in \M_0$ such that $P \notin \B(Q, 1)$ which contradicts the assumption that the diameter of $\M_0$ is 1. Therefore $\M_0 \subset \M$. Moreover, since every point $P \in \M_0$ lies on $\partial\Rr$ and $\M\subset\Rr$, $P$ must also lie on $\partial\M$.
	
	Before continuing with the proof, we state the following elementary geometric lemma (see Figure \ref{fig:distsphere}).
	
	\begin{lemma}\label{lem:spheres}
		Let $k<d$ and let $S$ be a $k$-dimensional sphere with center $O$ contained in a $(k+1)$-dimensional subspace $H$ of $\R^d$. Let $P$ be a point such that the projection $P_H$ of $P$ onto $H$ is not $O$. Then the distance from $P$ to a point $Q\in S$ increases as $Q$ moves on $S$ in the direction of the vector $O-P$ (or, equivalently $O-P_H$). To be precise, if $Q,Q'\in S$ are such that $(O-P)\cdot(Q'-Q)>0$, then $\norm{P - Q} < \norm{P - Q'}$.
		
		In particular, the maximum distance from $P$ to any point on $S$ is attained at the point $Q$ for which $O$ lies on the line segment $P_HQ$.
	\end{lemma}
	
	\begin{figure}
		\centering
		\includegraphics{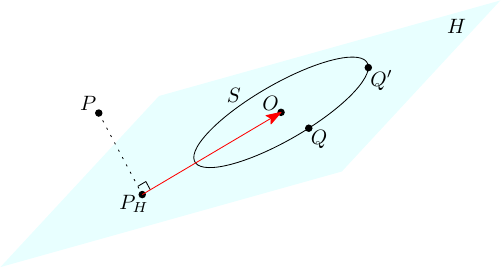}
		\caption{The distance from $P$ to $Q$ is less than the distance from $P$ to $Q'$.}
		\label{fig:distsphere}
	\end{figure}
	
	The proof is omitted as it is a straightforward exercise in Euclidean geometry.
	
	We now prove that the diameter of $\M_0$ is 1. Assume that $P,Q\in\M_0$ define a diameter of $\M_0$.
	Recall that $\M_0$ is the union of the six symmetrical components of the form $\F_{X,Y}$. We proceed by analyzing the possible locations of $P$ and $Q$ across these components.
	
	First, let us summarize the structure of a single component, $\F_{X,Y}$. As defined in Section \ref{sec:def}, $\F_{X,Y}$ is a region on a 2-dimensional sphere. Its boundary is composed of three vertices, $X$, $Y$, and $E$, which are connected by two arcs from the original Reuleaux simplex, $F_{X,E}$ and $F_{Y,E}$, and the arc $S_{X,Y}$. Any point in $\F_{X,Y}$ is either one of the vertices, on one of the boundary arcs, or in the relative interior of $\F_{X,Y}$, which we denote $\F^{\circ}_{X,Y}$.
	
	\noindent\textbf{Case 1: $P$ and $Q$ are in the same component.}
	Assume both $P$ and $Q$ belong to the same component, say $\F_{A,B}$. By construction, this component is a subset of $F_{A,B,E}$.
	This 2-face is a spherical Reuleaux triangle, it is contained in a sphere of radius $\sqrt{3}/2$ and has a diameter of exactly 1. In fact, using Lemma \ref{lem:spheres}, it is easy to verify that the diameters of $F_{A,B,E}$ are realized between any vertex and any point on its opposing boundary arc.
	Since $\F_{A,B} \subset F_{A,B,E}$, it immediately follows that the diameter of $\F_{A,B}$ is 1 and the diameter is achieved for pairs $\{P,Q\}=\{A,R\}$ where $R \in F_{B,E}$, and pairs $\{P,Q\}=\{B,S\}$ where $S \in F_{A,E}$.
	
	\noindent\textbf{Case 2: $P$ and $Q$ are in adjacent components.}
	Assume, without loss of generality, that $P\in\F_{A,B}$ and $Q\in\F_{B,C}$.
	Since these components are subsets of the positive cones $\pos(\{A,B,E\})$ and $\pos(\{B,C,E\})$ respectively, we may write
	\begin{align*}
		P &= \alpha_A A + \alpha_B B + \alpha_E E, \\
		Q &= \beta_B B + \beta_C C + \beta_E E,
	\end{align*}
	with all coefficients being non-negative.
	
	Consider the point $Q' = \beta_C C + \beta_B D + \beta_E E$. This point is simply the reflection of $Q$ with respect to the hyperplane $\spn(\{A,C,E\})$, since this reflection swaps the vertices $B$ and $D$. Note that $Q' \in \F_{C,D}$. A direct computation using the orthogonality of the basis vectors shows that:
	\[
	\norm{P-Q'}^2-\norm{P-Q}^2= \alpha_B \beta_B\ge 0.
	\]
	This implies that $\norm{P-Q} \le \norm{P-Q'}$. This inequality shows that for any point $P$ not on the boundary arc $F_{A,E}$ (i.e., $\alpha_B > 0$) and any point $Q$ not on the boundary arc $F_{C,E}$ (i.e., $\beta_B > 0$), the distance can be increased by reflecting $Q$.
	
	Therefore, any pair $(P,Q)$ that realizes the maximum distance between these two components cannot have both points in the relative interiors of their respective components. The maximum must be achieved when the equality $\norm{P-Q} = \norm{P-Q'}$ holds, which is if and only if $\alpha_B=0$ or $\beta_B=0$.
	\begin{itemize}
		\item If $\alpha_B=0$, then $P$ lies on $F_{A,E}$.
		\item If $\beta_B=0$, then $Q$ lies on $F_{C,E}$.
	\end{itemize}
	In either scenario, the problem reduces to finding the maximum distance between a point on a boundary arc and a point in a component that does not contain that arc. This problem is subsumed by the more general case of finding the distance between two points in opposite components, which we analyze next.
	
	\noindent\textbf{Case 3: $P$ and $Q$ are in opposite components.}
	Assume, without loss of generality, that $P\in\F_{A,B}$ and $Q\in\F_{C,D}$. We may write these points in their respective bases as
	\begin{align*}
		P &= \alpha_A A + \alpha_B B + \alpha_E E, \\
		Q &= \beta_C C + \beta_D D + \beta_E E,
	\end{align*}
	with all coefficients being non-negative.
	
	The component $\F_{C,D}$ lies on the 2-sphere $S$ with center $M_{A,B}$ and radius $\sqrt{3}/2$, which is contained in the 3-space $H=\spn(\{C,D,E\})$. The set $\{M_{A,B}, C, D\}$ forms an orthogonal basis for $H$. To apply Lemma \ref{lem:spheres}, we project $P$ onto $H$. The projection $P_H$ is given by:
	\begin{align*}
		P_H & = \frac{P \cdot M_{A,B}}{\norm{M_{A,B}}^2} M_{A,B} + \frac{P \cdot C}{\norm{C}^2} C + \frac{P \cdot D}{\norm{D}^2} D \\
		& = (\alpha_A+\alpha_B+\varphi\alpha_E) M_{A,B} + \frac{\varphi\alpha_E}{2} (C + D)\\
		& = (\alpha_A+\alpha_B+\varphi\alpha_E) M_{A,B} + \varphi\alpha_E M_{C,D}.
	\end{align*}
	From this expression, we see that $P_H$ lies in the plane $\Pi = \spn(\{M_{A,B},M_{C,D}\})$.
	
	\begin{figure}
		\centering
		\includegraphics{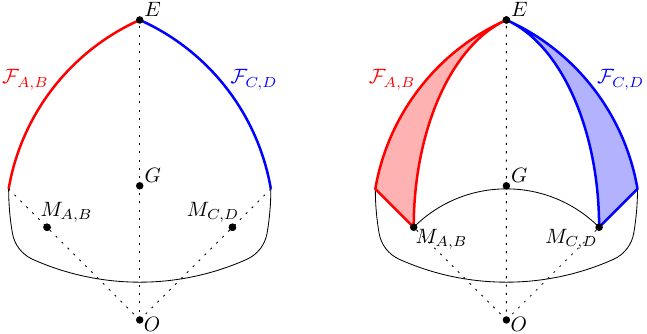}
		\caption{The sets $\Rr$, $\M$ and $\M_0$ cut by (left) and projected onto (right) the plane $\spn(\{M_{A,B},M_{C,D}\})$.}
		\label{fig:cut2d}
	\end{figure}
	
	We now consider two sub-cases for the location of $Q$.
	
	\noindent\textbf{Sub-case 3a: Either $P$ or $Q$ is in the relative interior of its component.}
	Assume $Q\in\F^{\circ}_{C,D}$, which implies $\beta_C, \beta_D, \beta_E > 0$.
	In this case, the distance between $P$ and $Q$ cannot increase when moving $Q$ in $S$.
	Since distance $\norm{P-Q}$ is assumed to be maximal, Lemma \ref{lem:spheres} states that the center of the sphere, $M_{A,B}$, lies on the line segment connecting $P_H$ and $Q$.
	
	We know that $P_H$ and $M_{A,B}$ both lie in the plane $\Pi$, thus $Q$ must also lie in this plane for the three points to be collinear.
    Since $Q=\beta_C C + \beta_D D + \beta_E E$ and $E\in\Pi$, then $Q$ lies in $\Pi$ if and only if $\beta_C = \beta_D$.
    In this case, $Q$ can be expressed in $\{M_{A,B},M_{C,D}\}$, the basis of $\Pi$ as:
	\[
	Q = (\varphi \beta_E) M_{A,B} + (2\beta_C+\varphi \beta_E) M_{C,D}.
	\]
	The condition that $M_{A,B}$ lies on the segment $P_HQ$ means $M_{A,B} = t P_H + (1-t) Q$ for some $t \in [0,1]$. Equating the coefficients in the basis $\{M_{A,B}, M_{C,D}\}$ yields the equations
	\begin{align}
		1 &= t(\alpha_A+\alpha_B+\varphi\alpha_E) + (1-t) \varphi \beta_E\quad \text{and}\label{eq:1}\\
		0 &= t \varphi\alpha_E + (1-t)(2\beta_C+\varphi \beta_E). \label{eq:2}
	\end{align}
	Since all coefficients in \eqref{eq:2} are non-negative, and since $2\beta_C+\varphi \beta_E > 0$, the equation can only hold if $t \varphi\alpha_E=0$ and $1-t=0$. This forces $t=1$ and $\alpha_E=0$.
	
	Substituting $t=1$ and $\alpha_E=0$ into \eqref{eq:1} gives $\alpha_A+\alpha_B = 1$. This implies $P=\alpha_A A+\alpha_B B$, meaning $P$ must lie on the segment connecting $A$ and $B$. For $P$ to also be in $\F_{A,B}$, it must lie on the boundary arc $S_{A,B}$. However, the only points of $S_{A,B}$ on this segment are precisely $A$ and $B$.
	
	In either case, since $Q\in\F_{C,D}\subset F_{C,D}$, it follows that $\norm{P-Q}=1$. The only pairs of points achieving this distance in this case are when $P=A$ or $P=B$ and $Q\in\F^{\circ}_{C,D}$.
	
	Analogously for $P\in\F^{\circ}_{A,B}$, if one point is in the relative interior of a component, the maximum possible distance to a point in an opposite component is 1.
	
	\noindent\textbf{Sub-case 3b: $P$ and $Q$ are on the boundaries of their respective components.}
	The boundary of each component consists of three arcs (e.g., $\partial\F_{A,B} = F_{A,E} \cup F_{B,E} \cup S_{A,B}$). By symmetry, we only need to analyze three representative configurations for the locations of $P$ and $Q$.
	
	\begin{description}[wide=\parindent, leftmargin=2\parindent]
		\item[Configuration 1] Assume $P \in F_{A,E}$ and $Q \in F_{C,E}$. We observe that both of these arcs are part of the boundary of a single component, namely $\F_{A,C}$. This reduces the problem to an application of Case 1, from which we know that the maximum distance is achieved between any pair of distinct vertices in $\{A,C,E\}$.
		
		\item[Configuration 2] Assume $P \in F_{A,E}$ and $Q \in S_{C,D}$.
		An application of Lemma \ref{lem:spheres}, similar to that of Sub-case 3a, shows that the maximum distance is again achieved when one of the points is a vertex.
		If $P=A$, then as a vertex of $\Rr$, its distance to any point $Q \in \F_{C,D}$ is 1. If $Q$ is an endpoint of its arc ($C$ or $D$), its distance to any point on $F_{A,E}$ is 1. In all other cases the distance is less than 1.
		
		\item[Configuration 3] Assume $P \in S_{A,B}$ and $Q \in S_{C,D}$. Since the arcs $S_{A,B}$ and $S_{C,D}$ are contained in orthogonal circles, the distance between any such $P$ and $Q$ is exactly 1.
	\end{description}
	All other combinations of boundary arcs are analogous to these cases by symmetry. In each case we have obtained a maximum distance of 1.
	
	\noindent\textbf{Conclusion.}
	We have shown through a case-by-case analysis that the distance between any two points $P, Q \in \M_0$ is at most 1. Since this maximum distance is achieved, the diameter of $\M_0$ is exactly 1. The preceding analysis also characterizes the pairs of points that realize this diameter.
	
	As established at the beginning of this section, this proves part (3) of Theorem \ref{thm:main}, namely that $\M_0 \subset \partial\M$.
	
	\subsection{A crucial property}
	This section is devoted to proving that for every boundary point $P \in \partial\M \setminus \M_0$, there exists a unique point $Q \in \M_0$ such that $\norm{P-Q}=1$.
	
	The existence of at least one such point $Q$ is a direct consequence of the construction of $\M$ and the compactness of $\M_0$. From the definition of $\M$, any point $P$ on the boundary $\partial\M$ must lie on the boundary of at least one unit ball centered at a point $Q\in\M_0$, which means $\norm{P-Q}=1$.
	
	To prove uniqueness, we show that for any two distinct points $Q, R \in \M_0$, the intersection of their corresponding unit spheres, $\Ss(Q,1) \cap \Ss(R,1)$, can only intersect the boundary $\partial\M$ at points in $\M_0$. This establishes that no point outside $\M_0$ can lie on two such spheres simultaneously. The proof requires a more detailed understanding of the geometry of such intersections, for which we first introduce the concept of a \emph{spindle}.
	
	\subsubsection{Spindles}
	The concept a spindle has appeared before in the literature (e.g., \cite{BLNP2007,BLN2024}). We present a more general definition tailored to our specific geometric setting.
	
	\begin{definition}
		Let $S$ be a relatively open subset of a $k$-sphere in $\R^d$ with radius $s$. For $r\ge s$, the \emph{spindle} generated by $S$ with radius $r$, denoted by $\Sp^k(S,r)$, is the intersection of all closed balls of radius $r$ centered at points in $S$. In other words,
		\[\Sp^k(S,r) = \bigcap_{x \in S} \B(x,r).\]
	\end{definition}
	
	The geometry of a spindle $\Sp^k(S,r)$ depends on the dimension $k$ of its generating set $S$ and the dimension $d$ of its ambient space. The case $k=0$, where $S$ consists of one or two points, simply yields a ball or the intersection of two balls and is not central to our analysis.
	
	The shape is easiest to visualize when $S$ is a complete $k$-sphere. In this situation, the resulting spindle is a body of revolution. Its axis of revolution is the $(d-k-1)$-dimensional subspace $J$ which is orthogonal to the $(k+1)$-space $H$ containing $S$. Consequently, the spindle inherits the symmetries of its generating sphere, meaning it is invariant under any isometry that leaves $S$ invariant.
	
	In $\R^2$, the only non-trivial case is a when $k=1$. If $S$ is a full circle, the resulting spindle is a disk.
	In $\R^3$, we have two possibilities. The case $k=1$, where $S$ is a circle, generates the classic lemon-shaped body  (see e.g., \cite{BLN2024}). The case $k=2$, where $S$ is a 2-sphere, generates a concentric ball.
	In higher dimensions, the geometry is richer, as each choice of $k$ from $1$ to $d-2$ produces distinct types of object.
	
	The following lemma provides a geometric characterization of the boundary of a spindle, two examples of which are illustrated in Figure \ref{fig:spindle}.
	
	\begin{lemma}\label{lem:spindle}
		Let $0<k<d$ and let $S$ be a relatively open subset of a $k$-sphere $\mathcal{S}_k$ that lies in a $(k+1)$-dimensional affine space $H$ of $\R^d$. Let $J$ be the orthogonal complement to $H$ passing through the center of $\mathcal{S}_k$. For any point $x \in S$, let $J_x$ be the affine space spanned by $J$ and $x$, and let $J_x^+$ be the closed half-space of $J_x$ bounded by $J$ that does not contain $x$. Then the points in $\Sp^k(S,1)$ at distance $1$ from $x$ are
		\[
		\Sp^k(S,1) \cap \Ss(x,1) = J_x^+ \cap \Ss(x,1).
		\]
	\end{lemma}
	
	\begin{figure}
		\centering
		\includegraphics{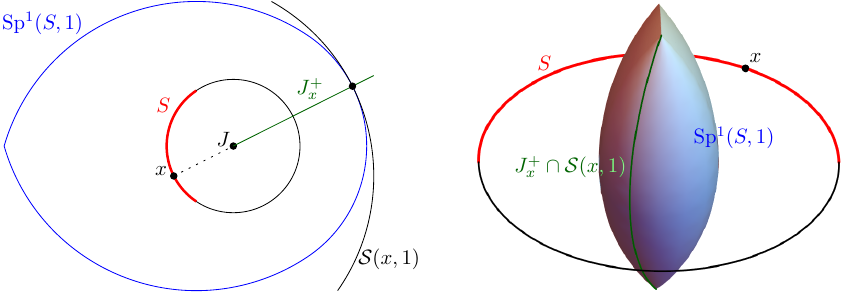}
		\caption{On the left is a spindle $\Sp^1(S,1)$ in $\R^2$, where $S$ is a subset of circle of radius $1/3$. On the right is another example of a spindle $\Sp^1(S,1)$ in $\R^3$ where $S$ is contained in a circle of radius $3/4$.}
		\label{fig:spindle}
	\end{figure}
	
	Note that $J_x^+ \cap \Ss(x,1)$ is also contained in the positive cone centered at $x$ generated by the vectors from $x$ to $J_x\cap \Ss(x,1)$.
	The proof of this lemma, which we omit for brevity, can be obtained by applying Lemma \ref{lem:spheres} on the sphere $\Ss_k$ and a point $P\in\Sp^k(S,1)\cap\Ss(x,1)$.

	The preceding lemma is useful because the body $\M$ is contained within spindles generated by subsets of $\M_0$.
	We may partition $\M_0$ into its constituent geometric parts:
	\begin{itemize}
		\item The vertices: $A$, $B$, $C$, $D$ and $E$.
		\item The open boundary arcs: sets of the form $F^\circ_{X,E}$ and $S^\circ_{X,Y}$.
		\item The relative interiors of the 2-dimensional components: sets of the form $\F^\circ_{X,Y}$.
	\end{itemize}
	Here, the superscript $\circ$ denotes the relative interior of a set.
	Consider $S \subset \M_0$ to be any one of these relatively open arcs or pieces of sphere. By the definition of $\M$, it follows that $\M \subset \Sp^k(S,1)$ (for $k=1$ or $2$). Consequently, for any point $x \in S$, a point $y \in \partial\M$ at distance 1 from $x$ must also lie within this spindle. Therefore, $y$ is confined to the intersection $\Sp^k(S,1) \cap \Ss(x,1)$. Lemma \ref{lem:spindle} provides a precise characterization of this region.
	
	The partition of $\M_0$ and the definition of $\M$ imply that the boundary set $\partial\M\setminus\M_0$ is composed of surface patches from spindles generated by the arcs and faces of $\M_0$, and the unit 3-spheres centered at the vertices $A$, $B$, $C$, $D$ and $E$.
	In the subsequent parts of this section it is shown that every point $P \in \partial\M\setminus\M_0$ belongs to exactly one of these surface patches. Furthermore, this uniqueness implies that the boundary of $\M$ outside of $\M_0$ is a smooth surface.
	
	\subsubsection{Proof of the Uniqueness Property}
	
	Assume there exists a point $P \in \partial\M$ that lies on the intersection of $\Ss(Q,1)$ and $\Ss(R,1)$, for two distinct points $Q,R \in \M_0$. Our goal is to show that $P\in\M_0$.
	
	The proof proceeds in three main steps. First, we analyze two configurations where $P$ is equidistant from a pair of vertices of the simplex $\Rr$. Second, we show that the more general case, where $Q$ and $R$ lie in the same component of $\M_0$, reduces to these initial vertex configurations. Finally, all remaining cases, where $Q$ and $R$ are in different components, are handled by using linear programming. Since the details are extensive, we focus on the key arguments and summarize the results of the computations.
	
	\noindent\textbf{Case 1: $A$ and $E$ are at distance 1 from $P$}
	Assume $\norm{P-A} = \norm{P-E} = 1$, then $P$ lies on the perpendicular bisecting hyperplane $\Pi_{A,E}$ of the segment $AE$.
	Since $\M \subset \Rr$, $P$ must also belong to the set of points in the Reuleaux simplex equidistant from $A$ and $E$, which is $F_{B,C,D}$. This face lies on a 2-sphere, which we denote $\Ss_{B,C,D}$, with center $M_{A,E}$ and radius $\sqrt{3}/2$. Note that the vertices $B$, $C$ and $D$ all lie on this sphere.

    Now, let $X$ be the midpoint of the arc $F_{A,E}$. Since $X \in \M_0$, $P$ is contained in the closed ball $\B(X, 1)$. Furthermore, the vertices $B$, $C$, and $D$ lie on the boundary of this ball. We claim that, apart from these vertices, $F_{B,C,D}$ lies strictly outside of $\B(X,1)$.
    
    To see this, consider the hyperplane $\Pi_{A,E}$ that contains the face $F_{B,C,D}$ and the points $M_{A,E}$ and $X$. Within this hyperplane, $\Ss_{B,C,D}$ and $\Ss(X, 1)$ intersect at the circle $W$. This circle passes through $B$, $C$, and $D$ and is centered at $M_{B,C,D}$.
    Note that $F_{B,C,D}$ intersects $W$ only at its vertices. Therefore, since the radius of $\Ss_{B,C,D}$ is smaller than the radius of $\Ss(X, 1)$, the only points from $F_{B,C,D}$ in $\B(X,1)$ are its vertices. We conclude that $\B(X,1)\cap F_{B,C,D}=\{B,C,D\}$. Thus, $P$ must be one of these vertices and is therefore in $\M_0$.
    
	\noindent\textbf{Case 2: $A$ and $B$ are at distance 1 from $P$}
	The argument is similar to the previous case. Assume $\norm{P-A} = \norm{P-B} = 1$. Then $P$ is on the perpendicular bisecting hyperplane $\Pi_{A,B}$ of the segment $AB$.
	Since $\M \subset \Rr$, $P$ must also lie on the face $F_{C,D,E}$. This face lies on the 2-sphere $\Ss_{C,D,E}$, centered at $M_{A,B}$ with radius $\sqrt{3}/2$.
	The vertices $C$, $D$ and $E$ all lie on this sphere.
	
	Let $X$ be the midpoint of the arc $S_{A,B}$. Since $X \in \M_0$, $P$ is contained in the closed ball $\B(X, 1)$. The boundary of this ball contains the vertices $C$ and $D$ and contains the entire arc $S_{C,D}$.
	
	The intersection of the face $F_{C,D,E}$ with the ball $\B(X, 1)$ is precisely the component $\F_{C,D}$. Thus, our point $P$ must belong to the set $\F_{C,D}$ which is contained in $\M_0$.
	
	\noindent\textbf{Case 3: $Q$ and $R$ are in the same component.}
	Assume that $Q$ and $R$ belong to the component $\F_{A,B}$. Let $\Ss_F$ be the 2-sphere that contains $\F_{A,B}$. Let $T = \Ss_F \cap \Ss(P,1)$, which is the set of all points on $\Ss_F$ at unit distance from $P$. Since $\Ss_F$ is a 2-sphere, there are three possibilities for $T$:
	
	\begin{description}[wide=\parindent, leftmargin=2\parindent]
		\item[$T$ is a single point] This is impossible, as $T$ must contain the two distinct points $Q$ and $R$.
		
		\item[$T$ is a circle] This circle divides the sphere $\Ss_F$ into two open spherical caps. By the definition of $\M$, every point $X \in \M_0$ must satisfy $\norm{P-X} \le 1$, therefore $\F_{A,B}$ must lie in the closed spherical cap where the distance to $P$ is less than or equal to 1.
		For $\F_{A,B}$ to be contained in this closed cap while touching its boundary circle $T$ at two distinct points, it is necessary that at least two of the vertices of $\F_{A,B}$ (i.e. $A$, $B$ and $E$) also lie on the circle $T$.
		
		\item[$T$ is the entire sphere $\Ss_F$] In this scenario, all points of $\Ss_F$ are at distance 1 from $P$. In particular, the three vertices $A$, $B$ and $E$ are at distance 1 from $P$.
	\end{description}
	
	In all cases, $P$ must be equidistant from at least two vertices of $\F_{A,B}$. This reduces the problem to the configurations already analyzed in Case 1 and Case 2.
	
	\noindent\textbf{Case 4: $Q$ and $R$ are not in a single component.}
	The strategy for the remaining cases is to define a region $\mathcal{C}$ that must contain $P$, based on the locations of $Q$ and $R$, and then to verify that the intersection $\mathcal{C} \cap \partial\M$ is a subset of $\M_0$.
	For each $Q$ and $R$, the regions we construct are polyhedral. This allows us to use linear programming to easily compute their intersection.
	
	\begin{figure}
		\centering
		\includegraphics{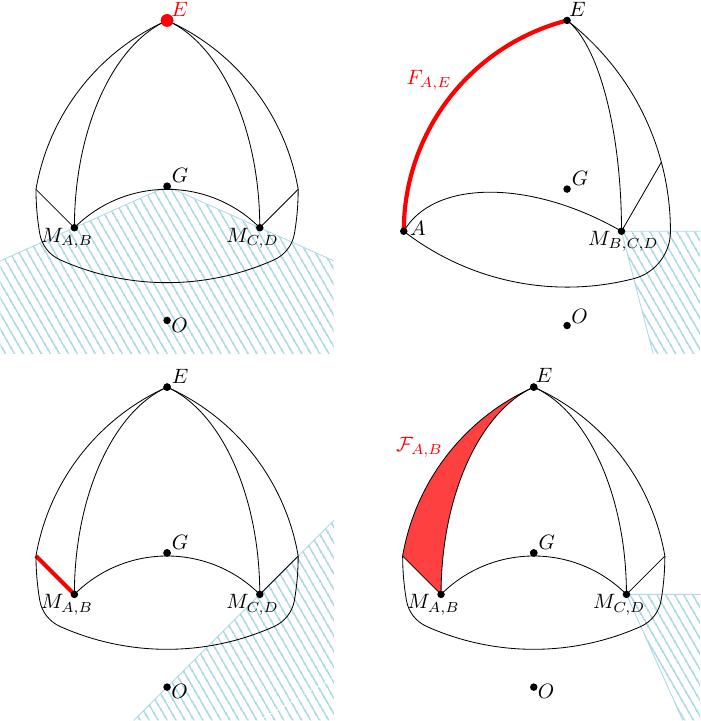}
		\caption{Projections of the set $\M$ onto a 2-dimensional plane. The highlighted red sets ($E$, $F_{A,E}$, $S_{A,B}$, and $\F_{A,B}$) are the projections of a vertex and sets that generate spindles. The blue hatched region is the resulting polyhedral region that must contain $P$.}
		\label{fig:projregions}
	\end{figure}
	
	Figure \ref{fig:projregions} provides a 2-dimensional illustration of these constraint regions. Below, we define these polyhedral regions for each type of generating set, up to symmetry.
	
	\begin{description}[wide=\parindent, leftmargin=2\parindent]
		\item[If $Q=A$]
		The sphere $\Ss(Q,1)$ intersects the Reuleaux simplex $\Rr$ only at points on the opposite face, $F_{B,C,D,E}$. This confines $P$ to the set:
		\begin{itemize}
			\item The positive cone with apex $G$ spanned by the vectors $B-G, C-G, D-G,$ and $E-G$.
		\end{itemize}
		
		\item[If $Q\in F_{A,E}^\circ$] By Lemma \ref{lem:spindle}, the sphere $\Ss(Q,1)$ can intersect $\M$ only within the region defined by the intersection of the following half-spaces:
		\begin{itemize}
			\item The closed half-space bounded by the hyperplane through $A,B,C,D$ that does not contain $E$.
			\item The closed half-space bounded by the hyperplane through $B,C,D,E$ that does not contain $A$.
		\end{itemize}
		
		\item[If $Q\in S_{A,B}^\circ$] Similarly, by Lemma \ref{lem:spindle}, the intersection of $\Ss(Q,1)$ with $\M$ is confined to the region defined by:
		\begin{itemize}
			\item The closed half-space bounded by the hyperplane through $A,C,D,O$ that does not contain $B$.
			\item The closed half-space bounded by the hyperplane through $B,C,D,O$ that does not contain $A$.
		\end{itemize}
		
		\item[If $Q\in \F_{A,B}^\circ$] Finally, by Lemma \ref{lem:spindle} on the 2-dimensional component $\F_{A,B}$, $\Ss(Q,1)$ can only intersect $\M$ in the intersection of:
		\begin{itemize}
			\item The closed half-space bounded by the hyperplane through $A,B,C,D$ that does not contain $E$.
			\item The closed half-space bounded by the hyperplane through $A,C,D,E$ that does not contain $B$.
			\item The closed half-space bounded by the hyperplane through $B,C,D,E$ that does not contain $A$.
		\end{itemize}
	\end{description}
	
	By exploiting the symmetries of $\M_0$, the remaining analysis can be reduced to the representative configurations enumerated below.
	\begin{enumerate}
		\item Let $Q=A$ (a base vertex). The possibilities for $R$ are:
		\begin{multicols}{2}
			\begin{enumerate}
				\item $R\in S^\circ_{B,C}$
				\item $R\in \F^\circ_{B,C}$
			\end{enumerate}
		\end{multicols}
		\item Let $Q \in F^\circ_{A,E}$ (an apex-connected arc). The possibilities for $R$ are:
		\begin{multicols}{2}
			\begin{enumerate}
				\item $R\in S^\circ_{B,C}$
				\item $R\in \F^\circ_{B,C}$
			\end{enumerate}
		\end{multicols}
		\item Let $Q \in S^\circ_{A,B}$ (a base arc). The possibilities for $R$ are:
		\begin{multicols}{2}
			\begin{enumerate}
				\item $R\in S^\circ_{B,C}$
				\item $R\in S^\circ_{C,D}$
				\item $R\in \F^\circ_{B,C}$
				\item $R\in \F^\circ_{C,D}$
			\end{enumerate}
		\end{multicols}
		\item Let $Q \in \F^\circ_{A,B}$ (a component interior). The possibilities for $R$ are:
		\begin{multicols}{2}
			\begin{enumerate}
				\item $R\in \F^\circ_{B,C}$
				\item $R\in \F^\circ_{C,D}$
			\end{enumerate}
		\end{multicols}
	\end{enumerate}

    For all but one of the listed configurations, we apply a directional minimization strategy to restrict the location of the polyhedral region $\mathcal C$ with respect to $\M$. Assuming $\mathcal C$ is non-empty, evaluating the minimum along an appropriately chosen direction leads to one of its vertices $X$.
    This direction is chosen so that the hyperplane orthogonal to it through $X$ separates $\M$ from the interior of $\mathcal C$.
    If $X$ is outside of $\M$, then $\mathcal C\cap\partial\M$ is empty. If $X$ is a vertex of $\M$, then $\mathcal C\cap\partial\M$ contains a single point which belongs to $\M_0$.
    
	For configurations 1(a), 1(b), 2(a), 2(b), 3(c) and 4(a), $D$ is a minimizer in the direction $D-E$. This implies that $P=D$, which is in $\M_0$.
	
	For configurations 3(b) and 3(d), the minimizers in the direction $-E$ are $O$ and $$-\frac{\varphi^2}{2 \sqrt{2}}(C+D),$$ respectively. Since both are at distance larger than 1 from $E$, we conclude that neither region intersects $\M$.
	
	The region corresponding to configuration 4(b) is empty.
	
	This leaves only configuration 3(a). In this case, our analysis indicates that P can be expressed as
	\[P=-\alpha_AA-\alpha_BB-\alpha_CC+\alpha_DD\]
	for some non-negative coefficients $\alpha_A,\alpha_B,\alpha_C,\alpha_D\ge 0$.    
	Let $Q=\cos(s)A+\sin(s)B$ and $R=\cos(t)C+\sin(t)B$ with $s,t\in(0,\pi/2)$.
	The conditions $\norm{P-Q}=1$ and $\norm{P-R}=1$ yield
	\begin{align*}
		1 &= \frac12(1+\alpha_A^2+\alpha_B^2+\alpha_C^2+\alpha_D^2)+\alpha_A\cos(s)+\alpha_B\sin(s), \\
		1 &= \frac12(1+\alpha_A^2+\alpha_B^2+\alpha_C^2+\alpha_D^2)+\alpha_C\cos(t)+\alpha_B\sin(t).
	\end{align*}
	Since $P$ is in $M$, its distance to any point in $M_0$ cannot exceed 1. Therefore, the above expressions must be local maxima with respect to $s$ and $t$ respectively.
	Setting their derivatives to zero gives
	\begin{align*}
		\alpha_B\cos(s)-\alpha_A\sin(s)=0,\\
		\alpha_B\cos(t)-\alpha_C\sin(t)=0.
	\end{align*}
	Analyzing this system for fixed $s$ and $t$ reveals that the only solutions satisfy $\alpha_A=\alpha_B=\alpha_C=0$ and $\alpha_D=\pm 1$, but since the coefficients are non-negative we have $\alpha_D=1$. This implies $P=D$, which is in $\M_0$.
	
	This analysis shows that no point $P \in \partial\M \setminus \M_0$ can be at unit distance from two distinct points $Q, R \in \M_0$. This proves the uniqueness property stated at the beginning of this subsection. With this proven, we are ready to complete the proof of Theorem \ref{thm:main}.
	
	\subsection{Smoothness of the boundary}
	A point $P \in \partial\M$ is smooth if it lies on the boundary of exactly one of the defining balls $\B(Q,1)$ for some $Q \in \M_0$. A non-smooth point must lie on the boundary of at least two such balls, say $\B(Q_1,1)$ and $\B(Q_2,1)$ for distinct $Q_1, Q_2 \in \M_0$ (see, e.g., \cite{KMP2010}). This would imply that $P \in \Ss(Q_1,1) \cap \Ss(Q_2,1)$. This contradicts the analysis done in the previous subsection. This proves part (4) of the theorem.
	
	\subsection{Endpoints of diameters}
	To prove part (5), let $PP'$ be any diameter of $\M$. If $P \in \M_0$, the claim is satisfied. If $P \in \partial\M \setminus \M_0$, we have established that there is a unique point $Q \in \M_0$ such that $\norm{P-Q}=1$. Furthermore, since $\M$ is smooth at $P$, the vector $P-P'$ must be parallel to $Q-P$.
	Since $\M$ is strictly convex and $\M_0\subset\partial\M$, this unique diametrically opposite point must be $P'$. Therefore the diameter of $\M$ is 1 and $P' \in \M_0$.
	
	\subsection{\texorpdfstring{$\M$}{M} has constant width}
	The preceding results establish that $\M$ is a compact convex body with a diameter of exactly 1. Furthermore, for every point $P \in \partial\M$, there exists another point $Q \in \partial\M$ such that $\norm{P-Q}=1$. From here we follow the approach taken in \cite{MR2016}. By a well-known theorem of Pál (see \cite{MMO2019}), any compact convex set with these two properties is a body of constant width. This completes the proof of Theorem \ref{thm:main}.
	
	\section{The Shadow of \texorpdfstring{$\M$}{M}} \label{sec:shadow}
	
	Let $H_E$ be the 3-dimensional subspace orthogonal to the vector $E$, and let $\pi: \R^4 \to H_E$ be the orthogonal projection. In this section, we analyze the geometry of the shadow body $\pi(\M)$. As the orthogonal projection of a body of constant width, $\pi(\M)$ is also a body of constant width in three dimensions (see Figure \ref{fig:shadow}).

    The body $\pi(\M)$ has been previously constructed. It is defined in \cite{ABPR2025} as a $3$-dimensional projection of the $4$-dimensional body introduced in \cite{ABNPR2025}. In what follows, we provide an independent description of $\pi(\M)$ in $\R^3$. We then prove that these two bodies coincide.
    
	\begin{figure}
		\centering
		\includegraphics{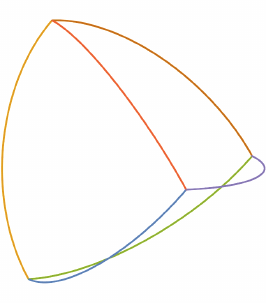}
		\caption{The elliptical arcs of the shadow of $\M$.}
		\label{fig:ellipses}
	\end{figure}
	
	Before we continue, we describe the projections of the six base arcs $S_{X,Y}$ (see Figure \ref{fig:ellipses}). Each arc projects onto an elliptical arc $\E_{X,Y}$ in the subspace $H_E$. This projected arc is the shorter arc of the outer Löwner–John ellipsoid of the rectangle with vertices $\pm\pi(X), \pm\pi(Y)$. This is, in other words, the unique ellipse of minimal area centered at the origin that passes through the points $\pi(X)$ and $\pi(Y)$. A direct calculation shows that each of these six ellipses has an eccentricity of $e = 1/\sqrt{2}$.
	
	For completeness, we provide explicit parametrizations for the six elliptical arcs $\E_{X,Y}$. To express these symmetrically, we choose an orthonormal basis for the subspace $H_E$ in which the projected vertices $\pi(A), \pi(B), \pi(C), \pi(D)$ are given by the coordinates:
	\[
	\frac{1}{2\sqrt{2}}(1,1,-1), \quad \frac{1}{2\sqrt{2}}(1,-1,1), \quad \frac{1}{2\sqrt{2}}(-1,1,1), \quad \text{and} \quad \frac{1}{2\sqrt{2}}(-1,-1,-1).
	\]
	In this basis, the parametrizations $\gamma_{X,Y}(t)$ for the arcs are given by:
	\begin{align*} \gamma_{A,B}(t)&=\frac{1}{2\sqrt{2}}\bigl(\cos(t)+\sin(t), \cos(t)-\sin(t), -\cos(t)+\sin(t)\bigr),\\ \gamma_{A,C}(t)&=\frac{1}{2\sqrt{2}}\bigl(\cos(t)-\sin(t), \cos(t)+\sin(t), -\cos(t)+\sin(t)\bigr),\\ \gamma_{A,D}(t)&=\frac{1}{2\sqrt{2}}\bigl(\cos(t)-\sin(t), \cos(t)-\sin(t), -\cos(t)-\sin(t)\bigr),\\ \gamma_{B,C}(t)&=\frac{1}{2\sqrt{2}}\bigl(\cos(t)-\sin(t), -\cos(t)+\sin(t), \cos(t)+\sin(t)\bigr),\\ \gamma_{B,D}(t)&=\frac{1}{2\sqrt{2}}\bigl(\cos(t)-\sin(t), -\cos(t)-\sin(t), \cos(t)-\sin(t)\bigr),\\ \gamma_{C,D}(t)&=\frac{1}{2\sqrt{2}}\bigl(-\cos(t)-\sin(t), \cos(t)-\sin(t), \cos(t)-\sin(t)\bigr), \end{align*}
	for $t \in [0, \pi/2]$.
	
	Let $\E$ be the union of these six projected elliptical arcs, that is,
	\[
	\E=\bigcup_{\substack{X,Y \in \{A,B,C,D\}}} \E_{X,Y}.
	\]
	We are now ready to give a characterization of the shadow body $\pi(\M)$.
	\begin{prop}
		The body $\pi(\M)$ is the intersection of all 3-dimensional unit balls centered at $\E$. That is,
		\[
		\pi(\M) = \bigcap_{p \in \E} \B^3(p,1),
		\]
		where $\B^3(p,1)$ denotes a 3-dimensional ball of radius 1 centered at $p$ in the subspace $H_E$.
	\end{prop}
	
	\begin{proof}
		A key property of the elliptical arcs $\E_{X,Y}$ is that for any point on one arc, there is a corresponding point on the opposite arc at unit distance. To be precise, consider a point $P = \alpha A + \beta B$ on the arc $S_{A,B}$ and its symmetrically corresponding point $Q = \alpha C + \beta D$ on the opposite arc $S_{C,D}$. The vector $P-Q$ has unit length and is parallel to the projection subspace $H_E$. This implies that since $\pi(\M)$ has constant width 1, the projected arcs $\E_{X,Y}$ must lie on its boundary.
		
		We also note that the projection of each 2-dimensional component $\F_{X,Y}$ is the set of points $\{\lambda p : p \in \E_{X,Y}, \lambda \in [0,1]\}$. This forms a planar sector of an ellipse: a region bounded by two line segments and the elliptical arc $\E_{X,Y}$.
		As a consequence, the points on $\pi(\M_0)\setminus\E$ are all interior points of $\pi(\M)$.
		
		By the properties of projections and intersections, and this last statement, we have the inclusion
		\begin{align*}
			\pi(\M)&=\pi\left(\bigcap_{p\in\M_0}\B(p,1)\right)\subset \bigcap_{p\in\M_0}\pi(\B(p,1))\\
			&=\bigcap_{p\in\pi(\M_0)} \B^3(p,1) = \bigcap_{p \in \bigcup \E} \B^3(p,1).
		\end{align*}
		We proceed by proving the reverse inclusion by contradiction.
		
		Assume there exists a point $P \in \bigcap_{p \in \pi(\M_0)} \B^3(p,1)$ such that $P \notin \pi(\M)$. This assumption implies that $\pi(\M_0)$ is contained in the closed unit ball $\B^3(P,1)$.
		
		Let $Q$ be the point in the set $\pi(\M)$ that is closest to $P$. Let $H_Q$ be the supporting half-space to $\pi(\M)$ at $Q$ that contains $\pi(\M)$ and whose exterior normal is parallel to the vector $P-Q$ (see Figure \ref{fig:shadowproof}). We now consider two cases for the location of $Q$.
		
		\begin{figure}
			\centering
			\includegraphics{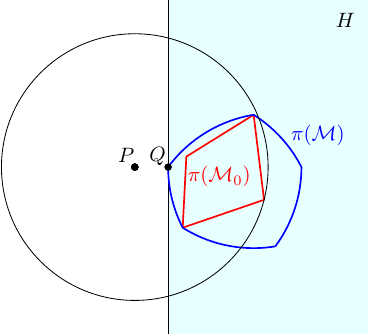}
			\caption{A representation of the relevant objects in the projection.}
			\label{fig:shadowproof}
		\end{figure}
		
		\noindent\textbf{Case 1: $Q$ is a smooth point of the boundary.}
		If $Q \in \partial\pi(\M) \setminus \pi(\M_0)$, then it is the projection of a smooth point $Q' \in \partial\M \setminus \M_0$. Since $\M$ is the intersection of unit balls, there is then a unique point $R \in \M_0$ such that $\norm{Q' - R}=1$. The vector $Q'-R$ is orthogonal to the supporting hyperplane of $\M$ at $Q'$ which is $\pi^{-1}(\partial H_Q)$, so $\pi(Q')-\pi(R) = Q - \pi(R)$. Thus, the points $P, Q,$ and $\pi(R)$ are collinear in that order and $\norm{P - \pi(R)} > 1$. This contradicts that $\pi(\M_0)$ must be contained in $\B^3(P,1)$.
		
		\noindent\textbf{Case 2: $Q$ is a non-smooth point of the boundary.}
		If $Q \in \pi(\M_0)$, then it must lie on one of the projected arcs, say $Q \in \E_{A,B}$. From our previous analysis, we know there exists another point $R \in \pi(\M_0)$ on the opposite arc $\E_{C,D}$ such that $\norm{Q-R}=1$. Our initial assumption requires that this point $R$ must lie inside the ball $\B^3(P,1)$ and in the supporting half-space $\mathcal{H}_Q$. However, there is no point in the set $\mathcal{H}_Q \cap \B^3(P,1)$ at distance of 1 from $Q$. This contradicts the existence of $R$.
		
		Since both cases lead to a contradiction, our initial assumption must be false, which completes the proof.
	\end{proof}

    The constant width body $\pi(\M)$ closely resembles the classical Reuleaux tetrahedron but is distinguished by its elliptical edges. A more precise analytical description of the surfaces that form the boundary in the neighborhood of these edges presents an interesting direction for future investigation.
    
    Now we prove that $\pi(\M)$ and the body $U_3$ defined in \cite{ABPR2025} coincide. In order to do this, we first recall the definition of the body $M_4$ from \cite{ABNPR2025}. While these bodies were originally constructed to have width $2$, we scale them to width $1$ to maintain consistency with our notation. Thus, $M_4$ is given by
    \begin{equation}
        M_4=\left\{v-w:v,w\in\R^4_+, \norm{v}^2+\left(\norm{w}+\frac{1}{\sqrt 2}\right)^2\le 1\right\}.
    \end{equation}
    Projecting orthogonally along the direction $E$ yields $U_3=\pi(M_4)$.

    \begin{prop}
        The constant width bodies $U_3$ and $\pi(\M)$ coincide.
    \end{prop}
    \begin{proof}
        First we note that every arc $S_{X,Y}$ is contained in $M_4$. Indeed, for every $v\in S_{X,Y}$ and $w=0$ we have
        \begin{equation*}
            \norm{v}^2+\left(\norm{w}+\frac{1}{\sqrt 2}\right)^2 = \frac{1}{2}+\left(\frac{1}{\sqrt{2}}\right)^2=1,
        \end{equation*}
        which implies that $v-w=v\in M_4$.

        It follows that $U_3=\pi(M_4)$ contains the same elliptical arcs $\E_{X,Y}$ that lie on the boundary of $\pi(\M)$. Since $U_3$ is a body of constant width $1$, it must be contained in $\bigcap_{p \in \E} \B^3(p,1)=\pi(\M)$. However, because $\pi(\M)$ is also a body of constant width $1$, the inclusion must be an equality. This shows that $U_3 = \pi(\M)$.
    \end{proof}

    After reading the previous proof, one may ask if $M_4$ and $\M$ also coincide, as they share the same symmetries and both contain the arcs $S_{X,Y}$. This is not the case. It is not difficult, for instance, to see that $E$ is not contained in $M_4$.

    \begin{table}
    \centering
    \begin{tabular}{|l| r @{.} l |}
        \hline
        \multicolumn{1}{|c|}{Body} & \multicolumn{2}{c|}{Volume} \\
        \hline
        Unit Sphere                & 0        & 5235987756 \\
        Reuleaux tetrahedron       & 0        & 4221577331 \\
        Average of Meissner bodies & 0        & 4208720183 \\
        Peabody                    & $\ge$ 0  & 4208257241 \\
        $U_3=\pi(\M)$              & 0        & 4204342424 \\
        Meissner bodies            & 0        & 4198600459 \\
        \hline
    \end{tabular}
    \caption{Some bodies of constant width $1$ and their volumes.}
    \label{tab:vols}
    \end{table}
    
    The authors of \cite{ABPR2025} posed an interesting open question. Among all bodies of constant width $1$ in $\R^3$ that possess the symmetry of the regular tetrahedron, which ones have minimal volume? In that same work, the authors numerically approximated the volumes of several $3$-dimensional bodies of constant width with tetrahedral symmetry. For completeness, we present these approximations in Table \ref{tab:vols}, together with the volumes of the Reuleaux tetrahedron and the Meissner bodies for comparison. We believe that the shadow body $U_3=\pi(\M)$ is a strong candidate for this minimal body.
     
	\section*{Acknowledgments}
	This work was supported by UNAM-PAPIIT IN111923. We are also grateful to Déborah Oliveros and Luis Montejano for fruitful discussions.

\end{document}